\documentclass{amsart}
\usepackage{graphicx}
\newtheorem{thm}{Theorem}[section]

\theoremstyle{definition}

\theoremstyle{remark}
\newtheorem{rem}[thm]{Remark}
\newtheorem{example}[thm]{Example}

\numberwithin{equation}{section}

\begin{document}

\title[]{A Global Hartman-Grobman Theorem }%
\author{X. Wang}%
\address{Department of Mathematics and Statistics\\ Texas Tech University\\Lubbock, TX 79409}%
\email{alex.wang@ttu.edu}%

\begin{abstract}
We showed that
for any bounded neighborhood of a hyperbolic equilibrium point  $x_0$, there is a transformation which is locally homeomorphism, such that
the system is changed into a linear system in this neighborhood.

If the eigenvalues of $Df(x_0)$ are all located in the left-half complex plane, then there is a homeomorphism on the whole region of attraction such that the nonlinear system on the region of attraction is changed into  a linear system under such a coordinate change.
\end{abstract}
\maketitle
\section{Introduction}

Reducing a nonlinear system
\begin{equation}\label{eq2}
\dot{x}=f(x)
\end{equation}
to a simpler form by choosing correct coordinates has always been a research direction. Since linear systems
are simplest and well studied, it is always a desire to change it to a linear one.
Poincar\'{e} in his dissertation showed that if $f$ is analytic at the equilibrium point $x_0$, and the eigenvalues of $Df(x_0)$
are nonresonant, then there is a formal power series of change of variable to change~(\ref{eq2}) to a linear system~\cite{poi90,arn88}.
Hartman and Grobman showed that if $f$ is continuously differentiable, then there is a neighborhood of a hyperbolic equilibrium point
and a homeomorphism on this neighborhood, such that the system in this neighborhood is changed to a linear system under such
a homeomorphism~\cite{gro59,har60,har64,per01}.

In this paper, we extend the Hartman-Grobman theorem to any bounded neighborhood of a hyperbolic equilibrium point, and show that
for any bounded neighborhood of a hyperbolic equilibrium point  $x_0$, there is a transformation which is locally homeomorphism, such that
the system is changed into a linear system in this neighborhood. Of course, such a transformation can not be a homeomorphism on this neighborhood because one could enlarge the neighborhood to include another equilibrium point. However we are able to show that
in the cases that either  all the eigenvalues of $Df(x_0)$ are located in the left-half complex plane, or in the right-half
complex plane, there is a homeomorphism such that the system on the region of attraction (or on the region of repulsion if the eigenvalues are located in the right-half complex plane) is change to a linea system.

\section{Hartman-Grobman Theorem on Any Bounded Region}

By applying a translation, we can always assume $0$ is the hyperbolic equilibrium point of~(\ref{eq2}).

\begin{thm} \label{thm1} Let $E$ be an open set of $\mathbb{R}^n$ containing the origin,  $f:E\to \mathbb{R}^n$ be a $C^1$ function on $E$, $0$ be a hyperbolic
equilibrium point of the system~(\ref{eq2}), and $N_{_M}=\{x: \|x\|<M\}$ be the neighborhood of the origin of radius $M$.
For any $M,\epsilon >0$ such that $\overline{N}_{_{M+\epsilon}}\subset E$, there exists a transformation $y=H(x)$, $H(0)=0$ and $H$ is a homeomorphism in a neighborhood of $0$, such that the system~(\ref{eq2})
is changed into the linear system
$$\dot{y}=Ay, \ \ \ \ A=Df(0)$$
in $N_{_M}$.
\end{thm}

\begin{proof}
Without lose of generality by applying a linear change of coordinates, assume $A=Df(0)$ has the form of
$$A=\left[\begin{array}{cc}
P & 0\\
0 & N\end{array}\right]$$
where the eigenvalues of $P$ are located in the right complex plane, and the eigenvalues of $N$ are located in the left complex plane.
Write
$$f(x)=\left[\begin{array}{cc}
P & 0\\
0 & N\end{array}\right]x+\left[\begin{array}{c}
W_1(x)\\
W_2(x)\end{array}\right]$$
where $\left[\begin{array}{c}
W_1(x)\\
W_2(x)\end{array}\right]=f(x)-Df(0)x$ has the properties that $W_i(0)=0$ and $DW_i(0)=0$, $i=1,2.$

For any $M,\epsilon >0$ such that $\overline{N}_{_{M+\epsilon}}\subset E$, let $\alpha(x)$ be a $C^\infty$ function such that
\begin{equation}\label{alpha}
\alpha(x)=\left\{\begin{array}{ll}
1,&\|x\|\leq M\\
0,&\|x\|\geq M+\epsilon
\end{array}\right. \ \ \ \mbox{and}\ \ \  \hat{W}_i(x)=W_i(\alpha(x)x),\ i=1,2.
\end{equation}
Then $W_i(x)=\hat{W}_i(x)$, i=1,2, for all $x$ with $\|x\|<M$.
Let $\phi_t(x)$ be the flow of the vector field $f(x)$, i.e.
$x(t)=\phi_t(x_0)$ is the solution of the initial value problem
$$\dot{x}=f(x),\ \ \ \  x(0)=x_0.$$
Define
\begin{equation} \label{map}
H(x)=x+\left[\begin{array}{c}
\int_0^\infty e^{-Ps}\hat{W}_1(\phi_s(x))\ ds\\ \\
-\int_{-\infty}^0 e^{-Ns}\hat{W}_2(\phi_s(x))\ ds
\end{array}\right].
\end{equation}
Since $\hat{W}_i$ are continuous on $\overline{N}_{_{M+\epsilon}}$, $\hat{W}_i(\phi_s(x))$ is bounded for all $s$. So the integrals in
the definition of $H$ converge, and $H$ is well defined.

Since $\hat{W}_i(0)=W_i(0)=0$ and  $D\hat{W}_i(0)=Dw_i(0)=0$, we have $DH(0)=I$, and by the inverse function theorem, $H$ is a local homeomorphism in a neighborhood of $0$.

For any solution $x(t)$ of the differential equation
$\dot{x}=f(x),$
\begin{eqnarray*}
H(x(t))&=&x(t)+\left[\begin{array}{c}
\int_0^\infty e^{-Ps}\hat{w}_1(\phi_s(x(t)))\ ds\\ \\
-\int_{-\infty}^0 e^{-Ns}\hat{w}_2(\phi_s(x(t)))\ ds
\end{array}\right]\\ \\
&=&x(t)+\left[\begin{array}{c}
\int_0^\infty e^{-Ps}\hat{w}_1(x(t+s))\ ds\\ \\
-\int_{-\infty}^0 e^{-Ns}\hat{w}_2(x(t+s))\ ds
\end{array}\right]\\ \\
&=&x(t)+\left[\begin{array}{c}
\int_t^\infty e^{-P(\tau-t)}\hat{w}_1(x(\tau))\ d\tau\\ \\
-\int_{-\infty}^t e^{-N(\tau-t)}\hat{w}_2(x(\tau))\ d\tau
\end{array}\right].
\end{eqnarray*}
So
under the transformation
$y=H(x)$
\begin{eqnarray*}
\dot{y}&=&\frac{d}{dt}H(x(t))\\
&=&\dot{x}-\left[\begin{array}{c}
\hat{w}_1(x(t))\\
\hat{w}_2(x(t))
\end{array}\right]+\left[\begin{array}{cc}
P & 0\\
0 & N\end{array}\right]\left[\begin{array}{c}
\int_t^\infty e^{-P(\tau-t)}\hat{w}_1(x(\tau))\ d\tau\\ \\
-\int_{-\infty}^t e^{-N(\tau-t)}\hat{w}_2(x(\tau))\ d\tau
\end{array}\right]\\ \\
&=&Ax+\left[\begin{array}{c}
w_1(x)\\
w_2(x)
\end{array}\right]-\left[\begin{array}{c}
\hat{w}_1(x)\\
\hat{w}_2(x)
\end{array}\right]+A\left[\begin{array}{c}
\int_0^\infty e^{-Ps}\hat{w}_1(\phi_s(x))\ d\tau\\ \\
-\int_{-\infty}^0 e^{-Ns}\hat{w}_2(\phi_s(x))\ d\tau
\end{array}\right]\\
&=& Ay
\end{eqnarray*}
for all $x\in N_{_M}$.
\end{proof}

It should be pointed out that the transformation $y=H(x)$ is not necessary a homeomorphism on $N_{_M}$.
Clearly $H$ maps all the equilibrium points of $\dot{x}=f(x)$ in $N_{_M}$ to $0$, because for any  equilibrium $x_0\in N_{_M}$
\begin{eqnarray*}
H(x_0)&=&x_0+\left[\begin{array}{c}
\int_0^\infty e^{-Ps}\hat{w}_1(\phi_s(x_0))\ ds\\ \\
-\int_{-\infty}^0 e^{-Ns}\hat{w}_2(\phi_s(x_0))\ ds
\end{array}\right]\\ \\
&=&x_0+\left[\begin{array}{c}
\int_0^\infty e^{-Ps}\ ds\ w_1(x_0)\\ \\
-\int_{-\infty}^0 e^{-Ns}\ ds\ w_2(x_0)
\end{array}\right]\\ \\
&=&x_0+\left[\begin{array}{c}
P^{-1} w_1(x_0)\\
N^{-1} w_2(x_0)
\end{array}\right]\\
&=&A^{-1}f(x_0)=0
\end{eqnarray*}

Theorem~\ref{thm1} is an existence type result, the proof is not constructive because it is generally impossible to find an analytic solution of a nonlinear system. Nevertheless, we give some examples of nonlinear systems with analytic solutions to demonstrate the transformations. Note that in the proof, the $\hat{w}_1,\hat{w}_2$ are introduced in order to guarantee the convergence of the integrals. If the integrals converge for either $s\to\infty$ or $s\to -\infty$, then we do not need replace $w_1,w_2$ with $\hat{w}_1,\hat{w}_2$

\begin{example}
Consider the system
$$\dot{\left[\begin{array}{c}x_1\\
x_2\\
x_3\end{array}\right]}=\left[\begin{array}{c}x_1+x_2^2\\
-x_2+x_3^2\\
-x_3\end{array}\right].$$
The flow of the vector field is given by
$$\phi_t(x)=\left[\begin{array}{c}e^t\left(x_1+\frac{x_2^3}{3}+\frac{x_3^2x_2}{6}+\frac{x_3^4}{30}\right)-
\frac{e^{-2t}}{3}(x_2^2+2x_3^2x_2+x_3^4)+
\frac{e^{-3t}}{2}(x_3^2x_2+x_3^4)-
\frac{e^{-4t}}{5}x_3^4\\
e^{-t}(x_2+x_3^2)-e^{-2t}x_3^2\\
e^{-t}x_3 \end{array}\right].$$
We have $w_1(x)=x_2^2$, $w_2(x)=x_3^2$, and $w_3(x)=0$. Since both
$$\int_0^\infty e^{-s} w_1(\phi_s(x))\ ds=\int_0^\infty e^{-s}(e^{-s}(x_2+x_3^2)-e^{-2s}x_3^2)^2\ ds=\frac{1}{3}x_2^2+\frac{1}{6}x_2x_3^2+\frac{1}{30}x_3^4$$
and
$$\int_0^\infty e^s w_2(\phi_s(x))\ ds=\int_0^\infty e^s(x_3e^{-s})^2\ ds=x_3^2$$
converge,
$$H(x)=\left[\begin{array}{c}x_1+\int_0^\infty e^{-s} w_1(\phi_s(x))\ ds\\
x_2+\int_0^\infty e^s w_2(\phi_s(x))\ ds\\
x^3
\end{array}\right]=\left[\begin{array}{c}x_1+\frac{1}{3}x_2^2+\frac{1}{6}x_2x_3^2+\frac{1}{30}x_3^4\\
x_2+x_3^2\\
x^3
\end{array}\right],$$
and the transformation is given by
\begin{eqnarray*}
y_1&=&x_1+\frac{1}{3}x_2^2+\frac{1}{6}x_2x_3^2+\frac{1}{30}x_3^4,\\
y_2&=&x_2+x_3^2,\\
y_3&=&x_3.
\end{eqnarray*}
We can verify
\begin{eqnarray*}
\dot{y}_1&=&\dot{x}_1+\frac{2}{3}x_2\dot{x}_2+\frac{1}{6}(2x_2x_3\dot{x}_3+\dot{x}_2x_3^2)
+\frac{2}{15}x_3^3\dot{x}_3   \\
&=&x_1+x_2^2+\frac{2}{3}x_2(-x_2+x_3^2)+\frac{1}{6}(2x_2x_3(-x_3)+(-x_2+x_3^2)x_3^2)+\frac{2}{15}x_3^3(-x_3)\\
&=&x_1+\frac{1}{3}x_2^2+\frac{1}{6}x_2x_3^2+\frac{1}{30}x_3^4=y_3\\
\dot{y}_2&=&\dot{x}_2+2x_3\dot{x_3}=-x_2+x_3^2+2x_3(-x_3)=-x_2-x_3^2=-y_2\\
\dot{y}_3&=&\dot{x}_3=-x_3=-y_3
\end{eqnarray*}
and the system becomes linear.

\end{example}

\begin{example}
Consider the system
$$
\begin{array}{lcl}
\dot{x}_1&=&x_1\\
\dot{x}_2&=&-x_2+x_1x_3^2\\
\dot{x}_3&=&-x_3.
\end{array}$$
The flow of the system is given by
$$\phi_t(x)=\left[\begin{array}{c}
e^tx_1\\
e^{-t}(x_2+x_1x_3^2 t)\\
e^{-t}x_3
\end{array}\right].$$
We have $w_1(x)=0$, $w_2(x)=x_1x_3^2$, and $w_3(x)=0$.
Both integrals
$$\int_0^\infty e^{s} w_2(\phi_s(x))\ ds=\int_0^\infty e^{s}(e^{s}x_1(e^{-s}x_3)^2\ ds$$
$$-\int_{-\infty}^0 e^{s} w_2(\phi_s(x))\ ds=-\int_{-\infty}^0 e^{s}(e^{s}x_1(e^{-s}x_3)^2\ ds$$
diverge, so we need to modify the $w_2$.

Instead of using a $C^\infty$ $\alpha(x)$ to construct $\hat{W}$, here we use a simpler but discontinuous modification of $W(x)$.
For any $M>0$, define
$$\hat{w}_2(x)=\left\{\begin{array}{ll}
w_2(x),&\mbox{if $|w_2(x)|\leq M$}\\ \\
0,&\mbox{otherwise.}\end{array}\right.$$
Then
$$\hat{w}_2(\phi_s(x))=\left\{\begin{array}{ll}
x_1x_3^2e^{-s},&\mbox{if $|x_1x_3^2|e^{-s}\leq M$}\\
0,&\mbox{if $|x_1x_3^2|e^{-s}> M$.}\end{array}\right.
=\left\{\begin{array}{ll}
x_1x_3^2e^{-s},&\mbox{if $s\geq -\ln \frac{M}{|x_1x_3^2|}$}\\
0,&\mbox{if $s<-\ln \frac{M}{|x_1x_3^2|}$.}
\end{array}\right.$$
and
\begin{eqnarray*}
H(x)&=&\left[\begin{array}{c}x_1\\
x_2-\int_{-\infty}^0 e^s \hat{w}_2(\phi_s(x))\ ds\\
x^3
\end{array}\right]\\
&=&\left[\begin{array}{c}x_1\\
x_2-x_1x_3^2\int_{-\ln \frac{M}{|x_1x_3^2|}}^0 \ ds\\
x^3
\end{array}\right]\\
&=&\left[\begin{array}{c}x_1\\
x_2-x_1x_3^2\ln \frac{M}{|x_1x_3^2|}\\
x^3
\end{array}\right].
\end{eqnarray*}

Note that $H(x)$ is continuous at the place where $|x_1x_3^2|=0$ and
$$\lim_{|x_1x_3^2|\to 0} H(x)=x.$$
On the other hand, the subspaces defined by $x_1=0$ and $x_3=0$ are all invariant subspaces of the system, and on these subspaces, the system is linear.

On the region defined by $0<|x_1x_3^2|$, we have
$$\frac{d}{dt}x_1x_3^2 =\dot{x}_1x_3^2+2x_1x_3\dot{x}_3=-x_1x_3^2$$
and
$$x_1x_3^2\frac{d}{dt}\ln \frac{M}{|x_1x_3^2|}=\left\{\begin{array}{ll}
-\frac{x_1x_3^2}{|x_1x_3^2|}\frac{d}{dt}x_1x_3^2& x_1x_3^2>0\\
\frac{x_1x_3^2}{|x_1x_3^2|}\frac{d}{dt}x_1x_3^2& x_1x_3^2<0
\end{array}\right\}=-\frac{d}{dt}x_1x_3^2=x_1x_3^2.$$
Therefore for
$y(t)=H(x(t))$,
\begin{eqnarray*}\dot{y}&=&\left[\begin{array}{c}\dot{x}_1\\ \\
\dot{x}_2-\frac{d}{dt}(x_1x_3^2)\ln \frac{M}{|x_1x_3^2|}-x_1x_3^2\frac{d}{dt}\ln \frac{M}{|x_1x_3^2|}\\ \\
\dot{x}^3
\end{array}\right]\\
&=&\left[\begin{array}{c}x_1\\ \\
-x_2+x_1x_3^2+x_1x_3^2\ln \frac{M}{|x_1x_3^2|}-x_1x_3^2\\ \\
-x_3
\end{array}\right]\\
&=&\left[\begin{array}{c}x_1\\ \\
-\left(x_2-x_1x_3^2\ln \frac{M}{|x_1x_3^2|}\right)\\ \\
-x^3
\end{array}\right]\\
&=&\left[\begin{array}{c}y_1\\
-y_2\\
-y_3
\end{array}\right]
\end{eqnarray*}

\end{example}

\section{Global Hartman-Grobman Theorem}

In this section, we consider the special case that all the eigenvalues of $Df(0)$ are located in
the left-half complex plane. In such cases, we can given a homeomorphism on the region of attraction of the
origin which changes the nonlinear system to a linear system inside the whole region of a traction.

By reverse the time, the same is also true if the eigenvalues of $Df(0)$ are located in
the right-half complex plane and region of attraction is replaced by the region of repulsion.

\begin{thm}
Under the condition of Theorem~\ref{thm1}, if all the eigenvalues of $Df(0)$ are located in the left-half complex plane,
then there is transformation $y=H(x)$, $H(0)=0$, $H$ is homeomorphism from the region
of attraction of the origin of~(\ref{eq2}) to $\mathbb{R}^n$, and  the system~(\ref{eq2})
is changed into the linear system
$$\dot{y}=Ay, \ \ \ \ A=Df(0)$$
 on the region of attraction of $0$ under $y=H(x)$.
\end{thm}

\begin{proof}
Let
$$W(x)=f(x)-Df(0)x$$
then $f(x)=Ax+W(x)$.

Consider the partial differential equations
\begin{equation}\label{eq1}
Ah(x)-W(x)=Dh(x)(Ax+W(x))
\end{equation}
$$h(0)=0,\ \ \ \  Dh(0)=0.$$
For a solution $h(x)$ of~(\ref{eq1}), if we define
$H(x)=x+h(x),$
then $H(x)$ is a local homeomorphism in a neighborhood of $0$ by the inverse function theorem.
Under the transformation
$y(t)=H(x(t))$,
$$\dot{y}=\dot{x}+Dh(x(t))\dot{x}=Ax+W(x)+Dh(x)(Ax+W(x))=Ax+W(x)+Ah(x)-W(x)=Ay$$
and the system becomes a linear system.

The characteristic equations of (\ref{eq1}) is
\begin{eqnarray}
\dot{x}&=&Ax+W(x)\label{eq3} \\
\dot{z}&=&Az-W(x)\label{eq4}
\end{eqnarray}
The invariant manifold $z=h(x)$ of the system whose tangent space at $0$ is the $x$-space gives us the solution of (\ref{eq1}). Note that (\ref{eq3}) is the original system and (\ref{eq4}) is a linear
system driven by the solution of (\ref{eq3}).

As in~(\ref{map}), if we define
$$\hat{h}(x)=-\int_{-\infty}^0 e^{-As}\hat{W}(\phi_s(x))\ ds$$
where $\phi_s(x)$ is flow of~(\ref{eq3}) and $\hat{W}(x)=W(\alpha(x)x)$ with the $\alpha(x)$ in~(\ref{alpha}), then
for any solution $x(t)$ of~(\ref{eq3}),
\begin{eqnarray*}
\hat{h}(x(t))&=&-\int_{-\infty}^0 e^{-As}\hat{W}(\phi_s(x(t)))\ ds\\
&=&-\int_{-\infty}^0 e^{-As}\hat{W}(x(t+s)))\ ds\\
&=&-\int_{-\infty}^t e^{-A(\tau-t)}\hat{W}(x(\tau)))\ d\tau
\end{eqnarray*}
and for $z(t)=\hat{h}(x(t))$
$$\dot{z}=-\hat{W}(x)-A\int_{-\infty}^t e^{-A(\tau-t)}\hat{W}(x(\tau)))\ d\tau=Az-\hat{W}(x).$$
So
$$z(t)=\hat{h}(x(t))$$
is a solution of~(\ref{eq4}) if $x(t)\in N_{_M}$.

For any solution $x(t)$ of~(\ref{eq3}) with the initial value $x(0)$ in the region of attraction, there exists a $t_0>0$ such that
$x(t)\in N_{_M}$ for all $t\geq t_0$, which means that $z(t)=\hat{h}(x(t))$ is a solution of~(\ref{eq4}) for all $t\geq t_0$.

For a solution of $x(t)$ of~(\ref{eq3}),
\begin{equation} \label{z}
z(t)=e^{(t-\rho)A}\hat{h}(x(\rho))-\int_\rho^t e^{(t-\sigma)A} W(x(\sigma))\ d\sigma
\end{equation}
is a solution of~(\ref{eq4}) for all $t\neq \rho$, and if $\rho>t_0$,
$z(t)=\hat{h}(x(t))$ in a neighborhood of $\rho$ because they all satisfy~(\ref{eq4}) with the same initial condition
 at $t=\rho$.

(\ref{z}) can be rewritten as
\begin{eqnarray*}
z(t)&=&-e^{(t-\rho)A}\int_{-\infty}^0 e^{-As}\hat{W}(\phi_s(x(\rho)))\ ds-\int_\rho^t e^{(t-\sigma)A} W(x(\sigma))\ d\sigma\\
&=&-\int_{-\infty}^0 e^{-(\rho+s-t)A}\hat{W}(\phi_s(x(\rho)))\ ds-\int_\rho^t e^{(t-\sigma)A} W(x(\sigma))\ d\sigma\\
&=&-\int_{-\infty}^{\rho-t} e^{-\tau A}\hat{W}(\phi_{\tau-t-\rho}(x(\rho)))\ d\tau+\int_0^{\rho-t} e^{-\tau A}
W(\phi_\tau(x(t)))\ d\tau\\
&=&-\int_{-\infty}^{\rho-t} e^{-\tau A}\hat{W}(\phi_{\tau}(x(t)))\ d\tau+\int_0^{\rho-t} e^{-\tau A}
W(\phi_\tau(x(t)))\ d\tau\\
&=&-\int_{-\infty}^0 e^{-\tau A}\hat{W}(\phi_{\tau}(x(t)))\ d\tau+\int_0^{\rho-t} e^{-\tau A}
\left(W(\phi_\tau(x(t)))-\hat{W}(\phi_\tau(x(t)))\right)\ d\tau
\end{eqnarray*}
Let $\rho\to\infty$
\begin{equation} \label{z2}
z(t)=-\int_{-\infty}^0 e^{-\tau A}\hat{W}(\phi_{\tau}(x(t)))\ d\tau+\int_0^{\infty} e^{-\tau A}
\left(W(\phi_\tau(x(t)))-\hat{W}(\phi_\tau(x(t)))\right)\ d\tau
\end{equation}
Note that both integrals in~(\ref{z2}) converge if $x(t)$ is a solution of~(\ref{eq3}) with initial condition $x(0)$ in the region of attraction. The first integral converges because that $\hat{W}(\phi_{\tau}(x(t)))$
is bounded for all $\tau\in (-\infty,0]$. The second integral converges because when $\tau$ is large enough, $\phi_{\tau}(x(t))\in N_{_M}$ and
$W(\phi_\tau(x(t)))-\hat{W}(\phi_\tau(x(t)))=0$.

If we define
\begin{equation}\label{h}
h(x)=-\int_{-\infty}^0 e^{-\tau A}\hat{W}(\phi_{\tau}(x))\ d\tau+\int_0^{\infty} e^{-\tau A}
\left(W(\phi_\tau(x))-\hat{W}(\phi_\tau(x))\right)\ d\tau,
\end{equation}
then $z(t)$ in~(\ref{z2}) has the form that $z(t)=h(x(t))$, i.e.
$h(x)$ defined in~(\ref{h}) is an invariant manifold of~(\ref{eq3}, \ref{eq4}). Also since $\phi_\tau(0)\equiv 0$ for any $\tau$, $W(0)=\hat{W}(0)=0$ and
$DW(0)=D\hat{W}(0)=0$,
$$h(0)=0,\ \ \ \  Dh(0)=0.$$
So $y=x+h(x):=H(x)$ satisfy $\dot{y}=Ay$ for any $x$ in the region of attraction of the origin of~(\ref{eq2}).

Let $U$ be a neighborhood of the origin such that $H$ is a homeomorphism on $U$, then for any $\hat{y}\in \mathbb{R}^n$, since the
linear system $\dot{y}=Ay$ is asymptotically stable, there is $t_0>0$ such that $y(t_0)\in H(U)$ for the solution $y(t)$ with initial condition $y(0)=\hat{y}$. Let $x(t)$ be the solution of~(\ref{eq2}) such that $x(t_0)=H^{-1}(y(t_0))$. Then $\hat{x}=x(0)$ is the unique point such that $H(\hat{x})=\hat{y}$ by the uniqueness of the initial value problem.
So $H$ is one-to-one and onto. 
The continuities of $H$ and $H^{-1}$ follow the continuity of the flows with respect to the initial points.
Let $\hat{\phi}_t(y)$ be the flow of $\dot(y)=Ay$. For any $x$ in the region of attraction, we can choose $t$ large enough such that 
$\phi_t(x)\in U$. Then
$$H(x)=\hat{\phi}_{-t}(H(\phi_t(x))).$$
So $H(x)$ is continuous on the region of attraction. Conversely for any $y$, we can choose $t$ large enough such that
$\hat{\phi}_t(y)\in H(U)$. Then
$$H^{-1}(y)=\phi_{-t}(H^{-1}(\hat{\phi}_t(y))).$$
So $H^{-1}(y)$ is continuous at any $y$.

\end{proof}

\begin{rem}
The expression~(\ref{h}) can be simplified under two cases.
If
$$\int_{-\infty}^0 e^{-\tau A} W(\phi_{\tau}(x))\ d\tau$$
converges, then we can take $M=\infty$, so $\hat{W}\equiv W$ and
\begin{equation}
h(x)=-\int_{-\infty}^0 e^{-\tau A}W(\phi_{\tau}(x))\ d\tau.
\end{equation}
On the other hand, if
$$\int_0^{\infty} e^{-\tau A} W(\phi_{\tau}(x))\ d\tau$$
converges, then we can take $M=\epsilon=0$, so $\hat{W}\equiv 0$ and
\begin{equation}\label{infcv}
h(x)=\int_0^{\infty} e^{-\tau A} W(\phi_{\tau}(x))\ d\tau.
\end{equation}
\end{rem}

\begin{example}
Consider the system
$$\dot{x}=-\sin x.$$
Clearly the region of attraction of $0$ is $(-\pi, \pi)$.
We have
\begin{equation}\label{dt}
dt = \frac{-1}{\sin x} dx.
\end{equation}
The flow $\phi=\phi_t(x)$ of the system satisfies
\begin{equation}\label{et}
\frac{\sin \phi_t(x)}{1-\cos \phi_t(x)}=\frac{\sin (x)}{1-\cos(x)} e^{t}=\cot\left(\frac{x}{2}\right) e^t.
\end{equation}
For this system, $A=-1$ and $W(x)=x-\sin (x)$. Applying the change of variable
$$t=\ln \left(\tan\left(\frac{x}{2}\right) \frac{\sin \phi}{1-\cos \phi}\right)$$
on~(\ref{infcv}), we have (see ~(\ref{dt}) and~(\ref{et}))
$$dt = \frac{-1}{\sin \phi} d\phi,\ \ \ e^t= \tan\left(\frac{x}{2}\right) \frac{\sin \phi}{1-\cos \phi},\ \
\phi_0(x)=x,\ \ \phi_\infty(x)=0$$
and (\ref{infcv}) becomes
$$h(x)=\int_0^{\infty} e^{\tau} (\phi_\tau(x)-\sin \phi_\tau (x))\ d\tau=\tan\left(\frac{x}{2}\right) \int_0^x \frac{\phi -\sin (\phi)}{1-\cos(\phi)} \ d\phi.$$
We can see that
$$y=H(x)=x+ \tan\left(\frac{x}{2}\right) \int_0^x \frac{\phi -\sin (\phi)}{1-\cos(\phi)} \ d\phi$$
maps the region of attraction $(-\pi, \pi)$ of $0$ onto $\mathbb{R}$.
\end{example}


\begin{thebibliography}{99}
\bibitem{arn88} V.I. Arnold, {\em Geometric Methods in the Theorey of Ordinary Differential Equations},
2nd ed., Springer-Verlag, New York, 1998.

\bibitem{gro59} D. Grobman, {\em Homeomorphisms of Systems of Differential Equations (Russian)}, Dokl. Akad., Nauk.,
Vol. 128, 1959, pp 880-881.

\bibitem{har60} P. Hartman, {\em A Lemma in the Theory of Structural Stability of
Differential Equations}, Proc. Amer. Math. Soc.,
Vol. 11, 1960, pp 610-620.

\bibitem{har64} P. Hartman, {\em Ordinary Differential Equations},
Wiley, New York, 1964.

\bibitem{per01} L. Perko, {\em Differential Equations and Dynamical Systems},
3rd ed., Springer-Verlag, New York, 2001.


\bibitem{poi90} H. Poincar\'{e}, {\em Sur le probl\'{e}me des trois corps et les \'{e}quations}, Dynamique Acta Math.,
Vol. 13, 1890, pp 1-270.

\bibitem{wan99} X. Wang and W. Dayawansa, {\em On Global Lyapnuv Functions of Nonlinear Autonumous System}, Procedings of the 38th
IEEE Conference on Decision \& Control, Phoenix, 1999.
\end{thebibliography}
\end{document}